\documentclass[10pt]{amsart}

\numberwithin{equation}{section} 

\usepackage{fancyhdr}
\usepackage{bm}
\usepackage{amssymb}
\usepackage{xypic}

\def\myLambda{{{\Z[{1\over 2},i]}}}

\newtheorem{thm}{Theorem}[section]
\newtheorem{cor}[thm]{Corollary}
\newtheorem{lem}[thm]{Lemma}
\newtheorem{prop}[thm]{Proposition}
\theoremstyle{definition}	

\newtheorem{rem}[thm]{Remark}

\def\goth{\mathfrak}

\def\a{\alpha}

\def\cA{\mathcal{A}}

\def\cE{\mathcal{E}}

\def\cO{\mathcal{O}}

\def\C{\mathbb{C}}

\def\FF{\mathbb{F}}

\def\QQ{\mathbb{Q}}
\def\R{\mathbb{R}}

\def\Z{\mathbb{Z}}

\def\Aut{\mathrm{Aut}}

\def\Hom{\mathrm{Hom}}

\def\Ker{\mathrm{Ker}}

\def\lim{\mathrm{lim}}

\newcommand{\F}{\mathbb{F}}

\date{\today} 
\begin{document}

\title{On the mod-$2$ cohomology of $SL_3(\Z[\frac{1}{2},i])$} 

\begin{abstract} 
Let $\Gamma=SL_3(\Z[\frac{1}{2},i])$, let $X$ be any mod-$2$ acyclic $\Gamma$-CW complex 
on which $\Gamma$ acts with finite stabilizers and let $X_s$ be the 
$2$-singular locus of $X$. We calculate the mod-$2$ cohomology of the Borel construction of $X_s$ 
with respect to the action of $\Gamma$. This cohomology coincides with the mod-$2$ cohomology 
of $\Gamma$ in cohomological degrees bigger than $8$ 
and the result is compatible with a conjecture of 
Quillen which predicts the structure of the cohomology ring $H^*(\Gamma;\FF_2)$. 
\end{abstract}
 
\author{Hans-Werner Henn}
\address{Institut de Recherche Math\'ematique Avanc\'ee,
C.N.R.S. - Universit\'e de Strasbourg, F-67084 Strasbourg,
France}
  
\maketitle

\section{Introduction}  

A major motivation for studying the mod-$2$ cohomology of $SL_3(\Z[\frac{1}{2},i])$ 
comes from a conjecture of Quillen (Conjecture 14.7 of [Q1]) which 
concerns the structure of the mod-$p$ cohomology of $GL_n(\Lambda)$ where $\Lambda$ 
is a ring of $S$-integers in a number field such that $p$ is invertible in $\Lambda$ and 
$\Lambda$ contains a primitive $p$-th root of unity $\zeta_p$. 
The conjecture stipulates that under these 
assumptions $H^*(GL_n(\Lambda);\Z/p)$ is free over the polynomial algebra $\Z/p[c_1,\ldots,c_n]$
where the $c_i$ are the mod-$p$ Chern classes associated to an embedding of $\Lambda$ 
into the complex numbers.  In the sequel we will denote this conjecture by $C(n,\Lambda,p)$. 

We will show in Theorem \ref{QC} that for $\Lambda=\Z[\frac{1}{2},i]$ conjecture 
$C(n,\Lambda,2)$ is equivalent to the existence 
of an isomorphism 
$$
H^*(GL_n(\myLambda);\FF_2)\cong \FF_2[c_1,\ldots,c_n]\otimes 
E(e_1,e_1',\ldots, e_{2n-1},e_{2n-1}') 
$$  
where the classes $c_i$ are the Chern classes of the tautological $n$-dimensional complex 
representation of $GL_n(\myLambda)$, $E$ denotes an exterior algebra and the classes 
$e_{2i-1},e_{2i-1}'$ are of cohomological degree $2i-1$ for $i=1,\ldots,n$.

Conjecture $C(n,\myLambda,2)$ is trivially true for $n=1$ and has been verified for $n=2$ 
in \cite{Weiss}. On the other hand, Dwyer's method in \cite{D} using \'etale approximations 
$X_n$ for the homotopy type of the $2$-completion of $BGL_n(\Z[\frac{1}{2}])$ 
and comparing the set of homotopy classes of $[BP,X_n]$ with that of $[BP,BGL_n(\Z[\frac{1}{2}])]$  
for suitable cyclic groups of order $2^n$ can be adapted to disprove $C(16,\myLambda,2)$. 
We will not dwell on this in this paper. However, we note that \'etale approximations can also be used 
to show that if $C(n,\Z[\frac{1}{2},i],2)$ fails then $C(2n,\Z[{1\over 2}],2)$ fails as well 
\cite{HLtoappear}. We also note that $C(n,\Z[{1\over 2}],2)$ is known to be true for $n=2$ 
by \cite{Mitchell} and $n=3$ by \cite{sl3} 
but is known to be false for $n=32$ by \cite{D} and even for $n\geq 14$ \cite{HLtoappear}.

In this paper we give a partial calculation of $H^*(SL_3(\myLambda);\F_2)$ 
and make a first step in an attempt to study conjecture $C(3,\myLambda,2)$. We propose the same 
strategy as the one which was used in the case of $SL_3(\Z[\frac{1}{2}])$. 
In a first step one uses a centralizer spectral sequence introduced in \cite{borel} 
in order to calculate the mod-$2$ Borel cohomology $H^*_G(X_s;\FF_2)$ 
where $X$ is any mod-$2$ acyclic 
$G$-CW complex on which a suitable discrete group $G$ acts with finite stabilizers 
and $X_s$ is the $2$-singular locus of $X$, i.e.  
the subcomplex consisting of all points for which the isotropy group 
of the action of $G$ is of even order.   
For $G=SL_3(\Z[\frac{1}{2}])$ this step was carried out in \cite{borel} and for 
$G=SL_3(\Z[\frac{1}{2},i])$ it is carried out in this paper. The precise form of $X$ 
does not really matter in this step. 

The second step involves a very laborious analysis of the relative mod-$2$ Borel cohomology 
$H^*_G(X,X_s;\FF_2)$ and of the connecting homomorphism for the Borel cohomology 
of the pair $(X,X_s)$. In the case of $G=SL_3(\Z[\frac{1}{2}])$ this was carried out by hand in \cite{sl3}. 
A by hand calculation looks forbidding in the case of $G=SL_3(\Z[\frac{1}{2},i])$ and this paper 
makes no attempt on such a calculation. However, we do make some comments on what is likely to be 
involved in such an attempt.  

Here are the main results of this paper. In these results 
the elements $b_2$ respectively $b_3$ are of degree $4$ resp. $6$. They are given as  
Chern classes of the tautological $3$-dimensional complex representation of $SL_3(\Z[\frac{1}{2},i])$. 
The indices of the other elements give their cohomological degrees. These elements come 
from Quillen's exterior cohomology classes in the cohomology of $GL_3(\FF_p)$ for suitable primes $p$,  
for example for $p=5$ (cf. section \ref{cent-coh} for more details). 
Furthermore $\Sigma^n$ denotes $n$-fold 
suspension so that $\Sigma^4\FF_2$ is a one dimensional $\FF_2$-vector space 
concentrated in degree $4$. 

\begin{thm}\label{SES} Let $\Gamma=SL_3(\Z[\frac{1}{2},i])$ and let $X$ be any mod-$2$ acyclic  
$\Gamma$-CW complex such that the isotropy group of each cell is finite. Then the centralizer spectral sequence of \cite{borel}
$$
\lim^s_{\cA_*(\Gamma)}H^tC_{\Gamma}(E);\FF_2)\Longrightarrow H^{s+t}_{\Gamma}(X_s;\FF_2)
$$ 
collapses at $E_2$ and gives a short exact sequence 
$$
0\to \Sigma^4\FF_2\oplus\Sigma^4\FF_2\oplus \Sigma^7\FF_2\to H^*_{\Gamma}(X_s;\FF_2)\to 
\FF_2[b_2,b_3]\otimes E(d_3,d_3',d_5,d_5')\to 0 
$$
in which the second map is a map of graded algebras.  
\end{thm}

Next let 
\begin{equation*}\label{psi}
\psi: H^*(\Gamma;\FF_2)=H^*_{\Gamma}(X;\FF_2)\to H^*_{\Gamma}(X_s;\FF_2) 
\to \FF_2[b_2,b_3]\otimes E(d_3,d_3',d_5,d_5') 
\end{equation*}
be the composition of the map induced by the inclusion $X_s\subset X$ 
and the epimorphism of Theorem \ref{SES}.   

\begin{thm}\label{mainthm} Let $\Gamma=SL_3(\Z[\frac{1}{2},i])$ and $X$ be as in the previous theorem. 

a) If $SD_3(\Z[\frac{1}{2},i])$ denotes the subgroup of diagonal matrices 
of $\Gamma$ then the target of $\psi$ 
can be identified with a subalgebra of $H^*(SD_3(\Z[\frac{1}{2},i]);\FF_2)$ and $\psi$ is induced by 
the restriction homomorphism $H^*(B\Gamma;\FF_2)\to H^*(SD_3(\Z[\frac{1}{2},i]);\FF_2)$. 

b) There exists a map of graded $\FF_2$-algebras 
$$
\varphi:\FF_2[c_2,c_3]\otimes E(e_3,e_3',e_5,e_5')\to H^*(\Gamma;\FF_2)
$$
with $e_i$ and $e_i'$ of degree $2i-1$ such that the  composition of $\varphi$ with $\psi$ is the 
isomorphism which sends $c_i$ to $b_i$, $i=2,3$, $e_i$ to $d_i$ and $e_i'$ to $d_i'$, $i=3,5$. 

c) The homomorphism $\psi$  is surjective in all degrees, an isomorphism in degrees $*>8$ 
and its kernel is finite dimensional in degrees $*\leq 8$.

\end{thm}

\smallskip
\begin{rem} In section \ref{RQC} we will discuss the relation of Theorem \ref{mainthm} with a 
conjecture of Quillen on the structure of the cohomology of $H^*(GL(n,\Lambda);\FF_2)$ for   
rings of $S$-integers $\Lambda$ in a number field satisfying suitable assumptions 
(cf. 14.7 of \cite{Q1}). This conjecture would hold in the case of $n=3$ and 
$\Lambda=\Z[\frac{1}{2},i]$ if  the maps $\psi$ and $\psi$ of part (b) of Theorem \ref{mainthm} turned 
out to be isomorphisms (cf. Proposition \ref{C3+main}). 
\end{rem}

\smallskip 
The following result is an immediate consequence of Theorem \ref{mainthm}. 

\begin{cor}\label{cor-main-thm} 
Let $\Gamma=SL_3(\Z[\frac{1}{2},i])$ and $X$ be as in Theorem \ref{SES}. 
Then the following conditions are equivalent. 

a) The restriction homomorphism $H^*(B\Gamma;\FF_2)\to H^*(SD_3(\Z[\frac{1}{2},i]);\FF_2)$ 
is injective and $H^*(B\Gamma;\FF_2)$ is isomorphic as a graded $\FF_2$-algebra to 
$\FF_2[b_2,b_3]\otimes E(d_3,d_3',d_5,d_5')$. 

b) There is an isomorphism 
$$
H^*_{\Gamma}(X,X_s;\FF_2)\cong \Sigma^5\FF_2\oplus\Sigma^5\FF_2\oplus \Sigma^8\FF_2 
$$ 
and the connecting homomorphism $H^*_{\Gamma}(X_s;\FF_2)\to H^{*+1}_{\Gamma}(X,X_s;\FF_2)$ 
is surjective. 
\qed 
\end{cor}
\smallskip

The paper is organized as follows. In section \ref{CentSS} we recall the centralizer spectral sequence 
and in section \ref{CSSGamma} we prove Theorem 1.1 and Theorem 1.2. 
In Section \ref{comm} we make some comments 
on step 2 of the program of a complete calculation of $H^*(\Gamma;\FF_2)$. 
Finally in  section \ref{RQC} we discuss the relation with Quillen's conjecture. 
\medskip

The author gratefully acknowledges numerous enlightening conversations with Jean Lannes 
over many years  on the topics discussed in this paper.

\section{The centralizer spectral sequence} \label{CentSS} 

We recall the centralizer spectral sequence introduced in \cite{borel}. 

Let $G$ be a discrete group and let $p$ be a fixed prime. Let $\cA(G)$ be the category 
whose objects  are the elementary abelian $p$-subgroups $E$ of $G$, 
i.e. subgroups which are isomorphic to $(\Z/p)^k$ for some integer $k$; 
if $E_1$ and $E_2$ are elementary abelian $p$-subgroups of $G$, 
then the set of morphisms from $E_1$ to $E_2$ in $\cA(G)$ consists precisely of those group 
homomorphisms $\a:E_1\to E_2$ for which there exists an element $g\in G$ with $\a(e)=geg^{-1}$ 
for all $e\in E_1$. Let $\cA_*(G)$ be the full subcategory of $\cA(G)$ whose objects are the 
non-trivial elementary abelian $p$-subgroups.

For an elementary abelian $p$-subgroup we denote its centralizer in $G$ by $C_G(E)$. Then the 
assignment $E\mapsto H^*(C_G(E);\FF_p)$ determines a functor from $\cA_*(G)$ to the category 
$\cE$ of graded $\FF_p$-vector spaces. The inverse limit functor is a left exact functor 
from the functor category $\cE^{\cA_*(G)}$ to $\cE$. 
Its right derived functors are denoted by $\lim^s$. The $p$-rank $r_p(G)$ 
of a group $G$ is defined as the supremum of all $k$ such that $G$ contains a subgroup isomorphic 
to $(\Z/p)^k$. 

For a $G$-space $X$ and a fixed prime $p$ 
we denote by $X_s$ the $p$-singular locus, i.e. the subspace of $X$  
consisting of points whose isotropy group contains an element of order $p$.
Let $EG$ be the total space of the universal principal $G$-bundle. 
The mod-$p$ cohomology of the Borel construction $EG\times_GX$ of a $G$ space $X$ 
will be denoted $H^*_G(X;\FF_p)$. The following result is a special case of part (a) 
of Corollary 0.4 of \cite{borel}. 

\begin{thm}\label{CSS} 
Let $G$ be a discrete group and assume there exists a finite dimensional mod-$p$ 
acyclic $G$-CW complex $X$ such that the isotropy group of each cell is finite.
Then there exists a cohomological second quadrant spectral sequence  
$$
E_2^{s,t}=\lim_{\cA_*(G)}^sH^t(C_G(E);\FF_p)\Longrightarrow H^{s+t}_G(X_s;\FF_p) 
$$ 
with $E_2^{s,t}=0$ if $s\geq r_p(G)$ and $t\geq 0$. 
\end{thm} 

\begin{rem}\label{edge} The edge homomorphism in this spectral sequence is a map of algebras 
$$
H^*_G(X_s;\FF_p)\to \lim_{\cA_*(G)}H^*(C_G(E);\FF_p)
$$ 
which is given as follows.  

Let $X^E$ be the fixed points for the action of $E$ on $X$. The $G$-action on $X$ restricts to an action 
of the centralizer $C_G(E)$ on $X^E$ and the $G$-equivariant maps 
$$
G\times_{C_G(E)}X^{E}\to X_s, \ \ \ (g,x)\mapsto gx \ . 
$$ for 
$E\in \cA_*(G)$ induce compatible maps in Borel cohomology 
$$
H^*_G(X_s;\FF_2)\to H^*_G(G\times_{C_G(E)}X^{E};\FF_2)\cong 
H^*_{C_G(E)}(X^{E};\FF_2)\cong H^*(C_G(E);\FF_2)
$$ 
which assemble to give the map to the inverse limit. 
Here we have used that by classical Smith theory $X^E$ 
is mod $p$-acyclic if $X$ is mod-$p$ acyclic and 
hence we get canonical isomorphisms 
$H^*_{C_G(E)}(X^{E};\FF_2)\cong H^*_{C_G(E)}(*;\FF_2)\cong 
H^*(C_G(E);\FF_2)$. 

Furthermore the composition 
\begin{equation}\label{comp}
H^*(G;\FF_p)=H^*_G(X;\FF_p)\to H^*_G(X_s;\FF_2)\to H^*(C_G(E);\FF_2)
\end{equation}
is induced by the inclusions $C_G(E)\to G$ as $E$ varies through $\cA_*(G)$. 
\end{rem}

In \cite{borel} we have used this spectral sequence in the case $p=2$ and $G=SL_3(\Z)$. 
Here we will use it in the case $p=2$ and $G=SL(3,\Z[{1\over 2},i])$. In both cases we have $r_2(G)=2$ 
and hence the spectral sequence collapses at $E_2$ and degenerates into a short exact sequence 
\begin{equation}\label{SES1}
0\to \lim_{\cA_*(G)}^1H^t(C_G(E);\FF_2)\to H^{t+1}_G(X_s;\FF_2)\to 
\lim_{\cA_*(G)}H^{t+1}(C_G(E);\FF_2)\to 0 \ . 
\end{equation}

\section{The centralizer spectral sequence for $SL_3(\Z[\frac{1}{2},i])$}\label{CSSGamma}

\medskip

\subsection{The Quillen category} Let $K$ be any number field, let $\cO_K$ be its ring of integers  and 
consider the ring of $S$-integers 
$\cO_K[\frac{1}{2}]$. Then, up to equivalence,  the Quillen category of 
$G:=SL_3(\cO_K[{1\over 2}])$ for the 
prime $2$ is independant of $K$. In fact, because $2$ is invertible every elementary abelian 
$2$-subgroup is conjugate to a diagonal subgroup, and hence $\cA_*(G)$ has a 
skeleton, say $\cA$, with exactly two objects, say $E_1$ and $E_2$ of rank 1 and 2, respectively. 
We take $E_1$ to be the subgroup generated by the diagonal matrix 
whose first two diagonal entries are $-1$
and whose third diagonal entry is $1$, and $E_2$ to be the subgroup of all diagonal matrices with 
diagonal entries $1$ or  $-1$ and determinant 1. 

The automorphism group of $E_1$ is trivial, of course, 
while $\Aut_{\cA}(E_2)$ is isomorphic to the group of all abstract automorphisms of 
$E_2$ which we can identify with ${\goth S}_3$, the symmetric group on three elements. There are three
morphisms from $E_1$ to $E_2$ and $\Aut_{\cA}(E_2)$ acts transitively on them.  

\subsection{The centralizers and their cohomology}\label{cent-coh} 
For the centralizers in $H:=GL_3(\cO_K[\frac{1}{2}])$ 
we find $C_H(E_1)\cong GL_2(\cO_K[\frac{1}{2}])\times GL_1(\cO_K[\frac{1}{2}])$ 
resp. $C_H(E_2)\cong D_3(\cO_K[\frac{1}{2}])$ if $D_n(\cO_K[\frac{1}{2}])$
denotes the subgroup of diagonal matrices in $GL_n(\cO_K[\frac{1}{2}])$. 
This implies 
$$
C_G(E_1)\cong GL_2(\cO_K[\frac{1}{2}]),\ \ \  
C_G(E_2)\cong D_2(\cO_K[\frac{1}{2}])
\cong \cO_K[\frac{1}{2}]^{\times}\times \cO_K[\frac{1}{2}]^{\times}\ . 
$$  

From now on we specialize to the case $K=\QQ_2[i]$ where we have  
$\cO_K[\frac{1}{2}]=\Z[\frac{1}{2},i]$. 
In this case the cohomology of the centralizers is explicitly known. In the sequel we abbreviate 
 $SL_3(\Z[\frac{1}{2},i])$ by $\Gamma$. 
 
\subsubsection{The cohomology of $C_{\Gamma}(E_2)$}\label{cent-coh2}
There is an isomorphism of groups 
\begin{equation*}\label{units}
\Z/4\times \Z\cong \Z[\frac{1}{2},i]^{\times}, \ \   (n,m)\mapsto i^n(1+i)^m
\end{equation*}
and therefore we get an isomorphism 
\begin{equation}\label{E2-centralizer} 
H^*(C_{\Gamma}(E_2);\FF_2)\cong H^*(\Z[\frac{1}{2},i]^{\times}\times \Z[\frac{1}{2},i]^{\times};\FF_2)
\cong  \FF_2[y_1,y_2]\otimes E(x_1,x_1',x_2,x_2')   
\end{equation}
with $y_1$ and $y_2$ in degree $2$ and the other generators in degree $1$. 
We agree to choose the generators so that $y_1$, $x_1$ and $x_1'$ come from the first factor 
with $x_1$ and $x_1'$ being the dual basis to the basis of 
$$
H_1(\Z[\frac{1}{2},i]^{\times};\FF_2)\cong \Z[\frac{1}{2},i])^{\times}/\big(\Z[\frac{1}{2},i])^{\times}\big)^2
\cong \Z/2\times\Z/2
$$  
given by the image of $i$ and $(1+i)$ in the 
mod-$2$ reduction of the abelian group $GL_1(\Z[\frac{1}{2},i])$ and $y_1$ coming from 
$H^2(\Z/4;\FF_2)$; likewise with $y_2$, $x_2$ and $x_2'$ coming from the second factor. 

\subsubsection{The cohomology of $C_{\Gamma}(E_1)$}\label{cent-coh1}
This cohomology has been calculated in \cite{Weiss}. In fact, from Theorem 1 of \cite{Weiss} we know  
\begin{equation}\label{n=2}
H^*(C_{\Gamma}(E_1);\FF_2)\cong H^*(GL_2(\Z[\frac{1}{2},i]);\F_2)\cong  
\FF_2[c_1,c_2]\otimes E(e_1,e_1',e_3,e_3') \ . 
\end{equation} 

In the sequel we give a short summary of this calculation. 
The classes $e_1$, $e_1'$, $e_3$ and $e_3'$ are pulled  back from Quillen's exterior classes 
$q_1$ and $q_3$ \cite{Q2} in 
\begin{equation}\label{n=2F5}
H^*(GL_2(\FF_5);\FF_2)\cong \FF_2[c_1,c_2]\otimes E(q_1,q_3)
\end{equation}
via two ring homomorphisms  
\begin{equation}\label{reduc}
\pi:\Z[\frac{1}{2},i]\to \FF_5 \ , \ \ \pi':\Z[\frac{1}{2},i]\to \FF_5 \ . 
\end{equation}
We choose $\pi$ such that $i$ is sent to $3$ and $\pi'$ such that $i$ is sent to $2$.  

Then consider the two commutative diagrams (with horizontal arrows induced by inclusion and 
vertical arrows induced by $\pi$ resp. $\pi'$) 
\begin{equation}\label{commdiag}
\begin{matrix} 
D_2(\Z[\frac{1}{2}]) & \to & GL_2(\Z[\frac{1}{2}]) \\
\\
\downarrow     &   & \downarrow \\
\\
D_2(\FF_5) & \to & GL_2(\FF_5) \ .  \\
\end{matrix}
\end{equation}

By abuse of notation we can write 
\begin{equation}\label{D2F5}
H^*(D_2(\FF_5);\FF_2)\cong H^*(\FF_5^{\times}\times \FF_5^{\times};\FF_2)\cong 
\FF_2[y_1,y_2]\otimes E(x_1,x_2)
\end{equation}
with $y_1\in H^2(\FF_5^{\times};\FF_2)$ and $x_1\in H^2(\FF_5^{\times};\FF_2)$ coming from the first 
factor and likewise with $y_2$ and $x_2$ coming from the second factor. 
Then these ring homomorphisms induce two homomorphisms 
$$
\pi^*,\pi'^*:
H^*(D_2(\FF_5);\FF_2))\to H^*(D_2(\Z[\frac{1}{2}]);\FF_2)
$$ 
which in term of the isomorphisms (\ref{D2F5}) and (\ref{E2-centralizer}) are 
explicitly given by 
\begin{equation}\label{pipi'}
\pi^*(y_i)=y_i=\pi'^*(y_i),\ \ \pi^*(x_i)=x_i,\ \ \pi'^*(x_i)=x_i+x_i' \ \ \text{for}\ \  i=1,2 \ . 
\end{equation}

The cohomology of $GL_2(\FF_5)$ is detected by restriction to the cohomology of diagonal matrices  and 
restriction is given explicitly as follows: 
\begin{equation}\label{rest} 
\begin{array}{ll}
c_1\mapsto y_1+y_2,\ \ c_2\mapsto y_1y_2,\ \ q_1\mapsto x_1+x_2    & q_3\mapsto y_1x_2+y_2x_1 \ . \\
\end{array}
\end{equation}

Then $e_1,e_1',e_3,e_3'$ are defined via   
\begin{equation}\label{def-e_i's} 
e_1=\pi^*(q_1),\ \ e_3=\pi^*(q_3),\ \ e_1'=\pi'^*(q_1),\ \ e_3'=\pi'^*(q_3) \ .
\end{equation} 
If $c_1$ and $c_2$ are the Chern classes of the tautological $2$-dimensional complex 
representation of $GL_2(\Z[\frac{1}{2}],i)$, 
then the restriction homomorphism from $H^*(GL_2(\Z[\frac{1}{2},i]);\FF_2)$ 
to the cohomology of the subgroup of diagonal matrices is injective and by using (\ref{commdiag}) 
and (\ref{rest}) we see that it is explicitly given by 
\begin{equation}\label{restriction}
\begin{array}{lll} 
c_1 \mapsto \ y_1+y_2 &  c_2\mapsto y_1y_2 \\
e_1\mapsto \ x_1+x_2 &   e_3\mapsto y_1x_2+y_2x_1 \\  
e_1'\mapsto \ x_1+x_1'+x_2+x_2'  &    e_3'\mapsto  y_1(x_2+x_2')+y_2(x_1+x_1')  \ . \\ 
\end{array}
\end{equation}

\subsubsection{Functoriality} We note that together with the isomorphisms (\ref{E2-centralizer}) and 
(\ref{n=2}) the restriction (\ref{restriction}) also describes the map  
$$
\a_*: 
H^*(C_{\Gamma}(E_1);\FF_2)
\to 
H^*(C_{\Gamma}(E_2);\FF_2)
$$ 
induced from the standard inclusion of $E_1$ into $E_2$. 

To finish the description of $H^*(C_{\Gamma}(-);\FF_2)$ as a functor on $\cA$ it remains to describe the 
action of the symetric group $Aut_{\cA}(E_2)\cong {\goth S}_3$ of rank $3$ on 
$H^*(C_{\Gamma}(E_2);\FF_2)\cong  \FF_2[y_1,y_2]\otimes\Lambda(x_1,x_1',x_2,x_2')$     
and because of the multiplicative structure we need it only on the generators. 

If $\tau\in \Aut_{\cA}(E_2)$ corresponds to permuting the factors in 
$C_{\Gamma}(E_2)\cong GL_1(\Z[\frac{1}{2},i])\times GL_1(\Z[\frac{1}{2},i])$ 
then 
\begin{equation}\label{tau-action}
\begin{array}{lll} 
\tau_*(y_1)= y_2 &  \tau_*(x_1)= x_2 &  \tau_*(x_1')= x_2' \\ 
\tau_*(y_2)=y_1  & \tau_*(x_2)=x_1 &  \tau_*(x_2')=x_1'  \\  
\end{array}
\end{equation}
and if $\sigma\in \Aut_{\cA}(E_2)$ corresponds to the cyclic permutation of the diagonal entries 
(in suitable order) then 
\begin{equation}\label{sigma-action}
\begin{array}{llll} 
\sigma_*(y_1)= y_2 &  \sigma_*(x_1)= x_2 & \sigma_*(x_1')= x_2' \\
\sigma_*(y_2)=y_1+y_2 & \sigma_*(x_2)=x_1+x_2 &  \sigma_*(x_2')=x_1'+x_2'  \ . \\
\end{array}
\end{equation}

\subsection{Calculating the limit and its derived functors} 

In Proposition 4.3 of \cite{borel} we showed that for any functor $F$ from $\cA$ 
to $\Z_{(2)}$-modules there is an exact sequence  
\begin{equation}\label{lim-exactsequence}
0\to {\lim}_{\cA}F\to F(E_1)\buildrel{\varphi}\over\longrightarrow  
\Hom_{\Z[{\goth S}_3]}(St_{\Z},F(E_2))\to {\lim}^1_{\cA}F\to 0 
\end{equation}
where $St_{\Z}$ is the $\Z[{\goth S}_3]$ module given by the kernel of the augmentation 
$\Z[{\goth S}_3/{\goth S}_2]\to \Z$, and if $a$ and $b$ are chosen to give an integral basis of $St_{\Z}$ 
on which $\tau$ and $\sigma$ act via 
\begin{equation}\label{St-action}
\begin{array}{ll}
\tau_*(a)=b & \tau_*(b)=a \\
\sigma_*(a)=-b & \sigma_*(b)=a-b \\
\end{array}
\end{equation} 
then $\varphi(x)(a)=\a_*(x)-(\sigma_*)^2\a_*(x)$ and 
$\varphi(x)(b)=\a_*(x)-\sigma_*\a_*(x)$ if $x\in F(E_1)$.  

Because in our case the functor takes values in $\FF_2$-vector spaces we can replace 
$\Hom_{\Z[{\goth S}_3]}$ by $\Hom_{\FF_2[{\goth S}_3]}$ and $St_{\Z}$ by its mod-$2$ reduction.  
The following elementary lemma is needed in the analysis of the third term in the exact 
sequence (\ref{lim-exactsequence}). 
\smallskip

\begin{lem}\label{Steinberg-lemma} {\ } 

a) Let $St$ be the $\FF_2[{\goth S_3}]$-module given as
the kernel of the augmentation $\FF_2[{\goth S}_3/{\goth S}_2] \to \FF_2$. 
The tensor product $St\otimes St$ decomposes as $\FF_2[{\goth S}_3]$-module canonically as 
$$
St\otimes St\cong \FF_2[{\goth S}_3/A_3]\oplus St 
$$ 
where $A_3$ denotes the alternating group on three letters. 
In fact, the decomposition is given by  
$$
St\otimes St\cong Im(id+\sigma_*+\sigma_*^2)\oplus Ker(id+\sigma_*+\sigma_*^2)
$$ 
and the first summand is isomorphic to $\FF_2[{\goth S}_3/A_3]$ while the second summand 
is isomorphic to $St$.

b) The tensor product $\FF_2[{\goth S}_3/A_3]\otimes St$ is isomorphic to $St\oplus St$. 
\end{lem}

\begin{proof} a) It is well known that $St$ is a projective $\FF_2[{\goth S}_3]$-module, hence 
$St\otimes St$ is also projective. It is also well known that every projective indecomposable 
$\FF_2[{\goth S}_3]$-module is isomorphic to either $St$ or $\FF_2[{\goth S}_3/A_3]$. 
Both modules can be distinguished by the fact that 
$e:=id+\sigma_*+\sigma_*$  acts trivially on $St$ and as the identity on $\FF_2[{\goth S}_3/A_3]$. 

Furthemore $e$ is a central idempotent in $\FF_2[{\goth S}_3]$ and hence each 
$\FF_2[{\goth S}_3]$-module $M$ 
decomposes as  direct sum of $\FF_2[{\goth S}_3]$-modules 
$$
M\cong \text{Im}(e:M\to M)\oplus \Ker(e:M\to M) \ . 
$$ 
An easy calculation shows that in the case of $St\otimes St$ both submodules are non-trivial and this 
together with the fact these submodules must be projective proves the claim. 

b) Again each of the factors in the tensor product is a projective $\FF_2[{\goth S}_3]$-module, 
hence the tensor product is a projective $\FF_2[{\goth S}_3]$-module. 
Because $\sigma$ acts as the identity on $\FF_2[{\goth S}_3/A_3]$ we see that the idempotent $e$ acts 
trivially on the tensor product and this forces the tensor product to be isomorphic to $St\oplus St$.  \end{proof} 

\smallskip
\begin{lem} The Poincar\'e series $\chi_2$ of  
$\Hom_{\FF_2[{\goth S}_3}](St,\FF_2[y_1,y_2]\otimes E(x_1,x_1',x_2,x_2'))$ is given by 
$$
\chi_2={2t^2(1+3t^2+3t^4+t^6)+2t(1+2t^2+2t^4+2t^6+t^8)\over (1-t^4)(1-t^6)} \ . 
$$
\end{lem} 
\smallskip

\begin{proof} The isomorphism of  (\ref{E2-centralizer}) is an isomorphism of 
$\FF_2[{\goth S}_3]$-modules where the action of ${\goth S}_3$ is given by (\ref{tau-action}) and 
(\ref{sigma-action}). In particular we see that 
$H^1(GL_1(\Z[\frac{1}{2},i])\times GL_1(\Z[\frac{1}{2},i]);\FF_2)$ is isomorphic to $St\oplus St$ 
generated by $x_1,x_1',x_2,x_2'$. The exterior powers of $H^1$ are given as 
$$ 
E^k(x_1,x_2,x_1',x_2')\cong E^k(St\oplus St)\cong 
\bigoplus_{j=0}^k E^jSt\otimes E^{k-j}St 
$$
and, because $E^k(St)$ is isomorphic to $\Sigma^k\FF_2$ if $k=0,2$, isomorphic to 
$\Sigma St$ if $k=1$, and trivially otherwise, we obtain  
$$
E^k(x_1,x_2,x_1',x_2')\cong 
\begin{cases} 
\Sigma^k\FF_2 & k=0,4 \\
\Sigma^k(St\oplus St) & k=1,3 \\
\Sigma^2\FF_2\oplus \Sigma^2(St\otimes St)\oplus 
\Sigma^2\FF_2   &k=2   \\ 
0      & k\neq 0,1,2,3,4  \\
\end{cases}  \\ 
$$ 
where $\FF_2$ denotes the trivial $\FF_2[{\goth S}_3]$-module whose 
additive structure is that of $\FF_2$.  

Therefore the Poincar\'e series $\chi_2$ of $\Hom_{\FF_2[{\goth S}_3]}(St,H^*(C_G(E_2);\FF_2))$ 
decomposes according to the decomposition of $\Lambda(x_1,x_2',x_1',x_2')$ as sum 
\begin{equation}\label{chi2}
\chi_2:=(1+2t^2+t^4)\chi_{2,0}+t^2\chi_{2,1}+2(t+t^3)\chi_{2,2} 
\end{equation} 
where  $\chi_{2,0}$ is the Poincar\'e series of 
$\Hom_{\FF_2[{\goth S}_3]}(St,\FF_2[y_1,y_2])$, $\chi_{2,1}$ is the Poincar\'e series of   
$\Hom_{\FF_2[{\goth S}_3]}(St,St\otimes St\otimes \FF_2[y_1,y_2])$ and $\chi_{2,2}$ 
is that of $\Hom_{\FF_2[{\goth S}_3]}(St,St\otimes \FF_2[y_1,y_2])$. 

Furthermore it is well known (and elementary to verify) 
that there is an isomorphism of $\FF_2[{\goth S}_3]$-modules  
$St\oplus St\oplus\FF_2[{\goth S}_3/A_3]\cong \FF_2[{\goth S}_3]$ and 
therefore an isomorphism 
\begin{align*}
\FF_2[y_1,y_2]\cong &\ 
\Hom_{\FF_2[{\goth S}_3]}(St\oplus St\oplus\FF_2[{\goth S}_3/A_3],\FF_2[y_1,y_2]) \\
\cong &\ \Hom_{\FF_2[{\goth S}_3]}(St,\FF_2[y_1,y_2])^{\oplus 2}\oplus \FF_2[y_1,y_2]^{A_3} \ .  
\end{align*}

Together with the elementary fact that the $A_3$-invariants $\FF_2[y_1,y_2]^{A_3}$ 
form a free module over $\FF_2[y_1,y_2]^{\goth S}_3\cong \FF_2[c_2,c_3]$ 
on two generators $1$ and $y_1^3+y_1y_2^2+y_2^3$ of degree $0$ resp. $6$ 
this implies 
$$
2\chi_{2,0}+{1+t^6\over (1-t^4)(1-t^6)}={1\over (1-t^2)^2}
$$ 
and hence  
\begin{equation}\label{chi20}
\chi_{2,0}={t^2\over (1-t^2)(1-t^6)}\ .  
\end{equation}

It is elementary to check that $St$ and $\FF_2[{\goth S}_3/A_3]$ are both self-dual 
$\FF_2[{\goth S}_3]$-modules and hence Lemma \ref{Steinberg-lemma} gives 
$$
St\otimes St^*\cong \FF_2[{\goth S}_3/A_3]\oplus St 
$$ 
and 
\begin{align*}
St\otimes St^*\otimes St^*\cong &\ (\FF_2[{\goth S}_3/A_3]\oplus St)\otimes St^* \\
\cong &\ (\FF_2[{\goth S}_3/A_3]\otimes St)\oplus (St\otimes St) \\
\cong &\ St\oplus St\oplus St\oplus \FF_2[{\goth S}_3/A_3]   \ . \\
\end{align*}

Therefore, if $\chi_{\FF_2[y_1,y_2]^{A_3}}$ denotes the Poincar\'e series of the $A_3$-invariants then  
\begin{equation}\label{chi21}
\chi_{2,1}=3\chi_{2,0}+\chi_{\FF_2[y_1,y_2]^{A_3}}
={3t^2\over (1-t^2)(1-t^6)}+{1+t^6\over (1-t^4)(1-t^6)}
={1+3t^2+3t^4+t^6\over (1-t^4)(1-t^6)} 
\end{equation}
\begin{equation}\label{chi22}
\chi_{2,2}=\chi_{2,0}+\chi_{\FF_2[y_1,y_2]^{A_3}}
={t^2\over (1-t^2)(1-t^6)}+{1+t^6\over (1-t^4)(1-t^6)}
={1+t^2+t^4+ t^6\over (1-t^4)(1-t^6)} \ . 
\end{equation}
Finally (\ref{chi2}), (\ref{chi20}), (\ref{chi21}) and (\ref{chi22})  give 
\begin{align*} 
\chi_2=&\ {(1+2t^2+t^4)t^2(1+t^2)+t^2(1+3t^2+3t^4+t^6)+2(t+t^3)(1+t^2+t^4+t^6)\over (1-t^4)(1-t^6)}  \\
=&\ {2t^2(1+3t^2+3t^4+t^6)+2t(1+2t^2+2t^4+2t^6+t^8)\over (1-t^4)(1-t^6)} \ ,  \\
\end{align*} 
and this finishes the proof. 
\end{proof}

Theorem \ref{SES} is now an immediate consequence of Theorem \ref{CSS} and the following result.  

\medskip

\begin{prop}\label{lim-lim1} Let $p=2$ and $\Gamma=SL_3(\Z[\frac{1}{2},i])$. 

a) There is an isomorphism of graded $\FF_2$-algebras  
$$
\lim_{\cA_*(\Gamma)}H^*(C_{\Gamma}(E);\FF_2)\cong \FF_2[b_2,b_3]\otimes 
E(d_3,d_3',d_5,d_5') \ . 
$$ 
Furthermore, if we identify this limit with a subalgebra of 
$H^*(C_{\Gamma}(E_1);\FF_2)\cong \FF_2[c_1,c_2]\otimes E(e_1,e_1',e_3,e_3')$ then 
\begin{align*}
\begin{array}{ll}
b_2=c_1^2+c_2\ \  & b_3=c_1c_2 \\
d_3=e_3\ \ & d_5=c_1e_3+c_2e_1 \\
d_3'=e_3'\ \  &d_5'=c_1e_3'+c_2e_1'\ . \\ 
\end{array}
\end{align*}

b) There is an isomorphism  of graded $\FF_2$-vector spaces 
$$
\lim_{\cA_*(\Gamma)}^1H^*(C_{\Gamma}(E);\FF_2)\cong 
\Sigma^3\FF_2\oplus\Sigma^3\FF_2\oplus \Sigma^6\FF_2  \ . 
$$

c) For any $s>1$ 
$$
\lim_{\cA_*(\Gamma)}^sH^*(C_{\Gamma}(E);\FF_2)=0 \ . 
$$
\end{prop}

\begin{proof} a) It is straightforward to check that the subalgebra of 
$\FF_2[c_1,c_2]\otimes E(e_1,e_1',e_3,e_3')$ generated by the elements 
$c_1^2+c_2$, $c_1c_2$, $e_3,e_3',c_1e_3+c_2e_1,c_1e_3'+c_2e_1'$  is isomorphic to 
the tensor product of a polynomial algebra on two generators $b_2$ and $b_3$ of degree $4$ and 
$6$ and an exterior algebra on $4$ generators $d_3$, $d_3'$, $d_5$ and $d_5'$ of degree $3$, $3$, $5$ 
and $5$. In fact, it is clear that $c_1^2+c_2$ and $c_1c_2$ 
are algebraically independant and the elements 
$e_3,e_3',c_1e_3+c_2e_1,c_1e_3'+c_2e_1'$ are exterior classes; their product is given as 
$c_2^2e_3e_3'e_1e_1'\neq 0$, and  this implies easily that the exterior monomials in these elements 
are linearly independant over the polynomial algebra generated by $c_1^2+c_2$, $c_1c_2$. 
From now on we identify $b_2$, $b_3$, $d_3$, $d_3'$, $d_5$ and $d_5'$ with $c_1^2+c_2$, $c_1c_2$, 
$e_3$, $e_3'$, $c_1e_3+c_2e_1$ and $c_1e_3'+c_2e_1'$.

Now we use the exact sequence (\ref{lim-exactsequence}) and the description of $\varphi$ 
to determine the inverse limit. Because $\a_*$ is injective, we see that if we identify 
$H^*(C_{\Gamma}(E_1);\FF_2)$ with its image in $H^*(C_{\Gamma}(E_2);\FF_2)$ 
then the inverse limit can 
be identified with the intersection of the image of $\a_*$ with the invariants in 
$\FF_2[y_1,y_2]\otimes E(x_1,x_1',x_2,x_2')$ with respect to the action of the cyclic group of order 
$3$ of $\Aut_{\cA}(E_2)\cong {\goth S}_3$ generated by $\sigma$. This action has been described in 
(\ref{sigma-action}) and with these formulas it is straightfoward to check that the elements 
\begin{equation}\label{lots}
\begin{array}{lll} 
b_2&=&y_1^2+y_1y_2+y_2^2 \\ b_3&=&y_1y_2(y_1+y_2)  \\
d_3&=&y_1x_2+y_2x_1 \\ d_5&=&(y_1+y_2)(y_1x_2+y_2x_1)+y_1y_2(x_1+x_2) =y_1^2x_2+y_2^2x_1\\  
d_3'&=&y_1(x_2+x_2')+y_2(x_1+x_1')  \\    
d_5'&=&(y_1+y_2)(y_1(x_2+x_2')+y_2(x_1+x_1'))+y_1y_2(x_1+x_1'+x_2+x_2') \\
       &=& y_1^2(x_2+x_2')+y_2^2(x_1+x_1') \ . \\ 
\end{array}
\end{equation}
all belong to the inverse limit. 

Now consider the following Poincar\'e series 
\begin{align*}
\begin{array}{ll}
\chi_0:= \sum_{n\geq 0}\dim_{\FF_2}(\FF_2[b_2,b_3]\otimes E(e_3,e_3',e_5,e_5')^n)t^n
={(1+t^3)^2(1+t^5)^2\over (1-t^4)(1-t^6)} \\ 
\\
\chi_1:= \sum_{n\geq 0}\dim_{\FF_2}H^n(C_{\Gamma}(E_1);\FF_2)t^n= 
{(1+t)^2(1+t^3)^2\over (1-t^2)(1-t^4)} \\
\\
\chi_2:= {2t^2(1+3t^2+3t^4+t^6)+2t(1+2t^2+2t^4+2t^6+t^8)\over (1-t^4)(1-t^6)} \ .\\ 
\end{array}
\end{align*}
Then we have the following identity 
$$
\chi_0+\chi_2-\chi_1={p\over (1-t^4)(1-t^6)}
$$ 
with 
\begin{align*}
p=&\ (1+t^3)^2(1+t^5)^2+2t^2(1+3t^2+3t^4+t^6)\\
&\ +2t(1+2t^2+2t^4+2t^6+t^8)-(1+t)^2(1+t^3)^2(1+t^2+t^4) \\
=&\  2t^3+t^6-2t^7-2t^9-t^{10}-t^{12}+2t^{13}+t^{16} 
=(2t^3+t^6)(1-t^4)(1-t^6)  
\end{align*}
and therefore 
\begin{equation}\label{chi-identity}
\chi_0+\chi_2=\chi_1+(2t^3+t^6)\ .    
\end{equation}
This together with the fact that $\lim_{\cA_*(\Gamma)}H^*(C_{\Gamma}(E);\FF_2)$ 
contains a subalgebra which is isomorphic to 
$\FF_2[b_2,b_3]\otimes\Lambda(d_3,d_3',d_5,d_5')$  
already  implies that the sequence
$$
0\to \FF_2[b_2,b_3]\otimes E(d_3,d_3',d_5,d_5')\to H^*(C_{\Gamma}(E_1);\FF_2)
\buildrel{\varphi}\over\longrightarrow 
\Hom_{\FF_2[{\goth S}_3]}(St,H^*(C_{\Gamma}(E_1);\FF_2))\to 0 
$$ 
in which the left hand arrow is given by inclusion 
is exact except possibly in dimensions $3$ and $6$.   

In order to complete the proof of a)  it is now enough to verify that 
in degrees $3$ and $6$ the inverse limit is not bigger than 
$\FF_2[b_2,b_3]\otimes E(d_3,d_3',d_5,d_5')$. 
We leave this straightforward verification to the reader.  

Then b) follows immediately from (a) together with (\ref{chi-identity}) and the exact sequence 
(\ref{lim-exactsequence}), and (c) follows from Theorem \ref{CSS} and the fact that $r_2(G)=2$. 
\end{proof} 

We can now give the proof of Theorem \ref{mainthm}. 

\begin{proof}  a) The exact sequence of Theorem \ref{SES} is obtained from the exact sequence 
(\ref{SES1}) via Proposition \ref{lim-lim1}. Therefore the epimorphism of Theorem \ref{SES} 
is the edge homomorphism of the centralizer spectral sequence. 
The result then follows from (\ref{comp}) by observing that we have identified the target of the 
edge homomorphism with the subalgebra $\FF_2[b_2,b_3]\otimes E(d_3,d_3',d_5,d_5')$
of $H^*(C_{\Gamma}(E_1);\FF_2)$ and by recalling that $C_{\Gamma}(E_1)$ is equal to 
$SD_3(\Z[\frac{1}{2},i])$.

b) The two ring homomorphisms $\pi, \pi':\Z[\frac{1}{2},i]\to \FF_5$ of (\ref{reduc}) 
determine  homomorphisms $SL_3(\Z[\frac{1}{2},i])\subset GL_3(\Z[\frac{1}{2},i])\to GL_3(\FF_5)$. 
By \cite{Q2} we have 
$$
H^*GL_3(\FF_5);\FF_2)\cong \FF_3[c_1,c_2,c_3]\otimes E(q_1,q_3,q_5) \ . 
$$ 
We get a well defined homomorphism of $\FF_2$-graded algebras    
$$
\varphi:\FF_2[c_2,c_3]\otimes E(e_3,e_3',e_5,e_5')\to H^*(\Gamma;\FF_2)
$$
by sending $c_i$ to the $i$-th Chern class of the tautological $3$-dimensional representation of 
$\Gamma$ and by declaring $\varphi(e_i)=\pi^*(q_i)$ and $\varphi(e_i')=\pi'^*(q_i')$ for $i=3,5$.  
The classes $q_1$ resp. $q_3$ resp.  $q_5$ are the symmetrisations of $x_1$ resp. $y_1x_2$ 
resp. $y_1y_2x_3$ with respect to the natural action of ${\goth S}_3$ on 
$H^*(GL_3(\FF_5);\FF_2)\cong \FF_2[y_1,y_2,y_3]\otimes E(x_1,x_2,x_3)$  
(cf. (\ref{e-classes})  below). 

Next we determine the composition $\psi\phi$. The universal Chern classes $c_i$ are the elementary 
symmetric polynomials in variables, say $y_i$, and the inclusion 
$GL_2(\C)\subset SL_3(\C)\subset GL_3(\C)$ imposes  the relation $y_1+y_2+y_3=0$.  
This implies that the behaviour of $\psi$ on Chern classes is given by  
$$
c_1\mapsto 0,\ \ c_2\mapsto c_1^2+c_2=y_1^2+y_1y_2+y_2^2=b_2,\ \ c_3\mapsto 
c_1c_2=y_1y_2(y_1+y_2)=b_3\ . 
$$
In these equations we have identified  $H^*(GL_2(\Z[\frac{1}{2},i];\FF_2)$, as in the proof of Proposition 
\ref{lim-lim1}, via restriction with a subalgebra of $\FF_2[y_1,y_2]\otimes E(x_1,x_1',x_3,x_3')$. 

In order to determine the composition $\psi\varphi$ on the classes $e_3$, $e_3'$, $e_5$ and $e_5'$  we 
calculate at the level of $\FF_5$ and use naturality with respect to the homomorphisms induced by 
$\pi$ and $\pi'$. In fact, the inclusion 
$$
j:GL_2(\FF_5)\subset SL_3(\FF_5)\subset GL_3(\FF_5)
$$ 
induce in cohomology a map 
$$
\FF_3[c_1,c_2,c_3]\otimes E(q_1,q_3,q_5)\to \FF_2[c_1,c_2]\otimes E(e_1,e_3)
\subset  \FF_2[y_1,y_2]\otimes E(q_1,q_3) 
$$ 
which is easily determined from (\ref{e-classes}) below by imposing the relations $y_1+y_2+y_3=0$ and 
$x_1+x_2+x_3=0$ on the symmetrisation of the classes $y_1x_2$ resp. $y_1y_2x_3$ 
with respect to the natural action of ${\goth S}_3$ on the cohomology of diagonal matrices
$H^*(D_3(\FF_5);\FF_2)\cong \FF_2[y_1,y_2,y_3]\otimes E(x_1,x_2,x_3)$. 
Explicitly we get 
\begin{align*}
c_1\mapsto 0,\ \  c_2\mapsto y_1^2+y_1y_2+y_2^2,\ \ 
c_3\mapsto y_1y_2(y_1+y_2)  \\
q_1\mapsto 0, \ \ q_3\mapsto y_1x_2+y_2x_1,\ \ q_5\mapsto y_1^2x_2+y_2^2x_1 
\end{align*}
and if $i$ denotes the inclusion 
$$
GL_2(\Z[\frac{1}{2},i])\subset SL_3(\Z[\frac{1}{2},i])\subset GL_3(\Z[\frac{1}{2},i])
$$
then (\ref{pipi'}) and (\ref{lots}) imply  
\begin{equation*}
\begin{array}{lll lll lll} 
\psi(\varphi(e_3))&=&i^*(\pi^*(q_3))&=&\pi^*j^*(q_3)&=&\pi^*(y_1x_2+y_2x_1)&=&d_3 \\
\psi(\varphi(e_5))&=&i^*(\pi^*(q_5))&=&\pi^*j^*(q_5)&=&\pi^*(y_1^2x_2+y_2^2x_1)&=&d_5 \\ 
\psi(\varphi(e_3'))&=&i^*(\pi'^*(q_3))&=&\pi'^*j^*(q_3)&=&\pi'^*(y_1x_2+y_2x_1)&=&d_3' \\
\psi(\varphi(e_5'))&=&i^*(\pi'^*(q_5))&=&\pi'^*j^*(q_5)&=&\pi'^*(y_1^2x_2+y_2^2x_1)&=&d_5' \\ 
\end{array}
\end{equation*}
where we have identified the target of $\psi$ with a subalgebra of $H^*(GL_2(\Z[\frac{1}{2},i];\FF_2)$ 
and the latter via restriction with a subalgebra of $\FF_2[y_1,y_2]\otimes E(x_1x_1',x_3,x_3')$. 

c) The space $X$ can be taken to be the product of symmetric space $X_{\infty}:=SL_3(\C)/SU(2)$ 
and the Bruhat-Tits building $X_2$ for $SL_3(\QQ_2[i])$. Now 
$SL_3(\QQ_2[i])\backslash X_2$ is a $2$-simplex (cf. \cite{Buildings})  
and the projection map $X\to X_2$ induces a map 
$$
SL_3(\QQ_2[i])\backslash X\to SL_3(\QQ_2[i])\backslash X_2
$$
whose fibres have the homotopy type of a $6$-dimensional 
$SL_3(\Z[\frac{1}{2},i])$-invariant deformation retract (cf. section \ref{comm}). 
Therefore we get $H^n_G(X,X_s;\FF_2)=0$ 
if $n>8$ and the inclusion $X_s\subset X$ induces an isomorphism 
$H^n_G(X;\FF_2) \cong H^n_G(X_s;\FF_2)$ if $n>8$.  Then part c) simply follows from a) except for the 
finiteness statement for the kernel for which we refer to (\ref{E1}) and (\ref{E1-finite}) below.
\end{proof}

\section{Comments on step 2}\label{comm}  

The situation for $p=2$ and $G=SL_3(\Z[{1\over 2},i])$ is analogous to the situation for 
$p=2$ and 
$G=SL_3(\Z[{1\over 2}])$ for which step 2 was carried out in \cite{sl3} 
via a detailed study of the relative cohomology $H^*_G(X,X_s;\FF_2)$ for $X$ equal to 
the product of the symmetric space $X_{\infty}:=SL_3(\R)/SO(3)$  with the Bruhat-Tits building $X_2$ 
for $SL_3(\QQ_2)$; the spaces involved had a few hundred cells and the calculation was painful.  
In the case of $SL_3(\Z[\frac{1}{2},i])$ with $X$ the  product of $SL_3(\C)/SU(3)$ 
with the Bruhat-Tits building for  $SL_3(\QQ_2[i])$ the calculational complexity of the second step is much 
more involved and an explicit calculation by hand does not look feasible. 
However, in recent years there have been a lot of machine aided calculations of the cohomology of 
various arithmetic groups (for example \cite{GGetc}, \cite{BR}) and a machine aided calculation 
seems to be within reach. 

The natural strategy for undertaking this second step is to follow the same path as in \cite{sl3}. 
The equivariant cohomology $H^*_{\Gamma}(X,X_s;\FF_2)$ can be studied via the spectral sequence 
of the projection map 
$$
p: X=X_{\infty}\times X_2\to X_2\ . 
$$ 
This gives a spectral sequence  with 
\begin{equation}\label{E1}
E_1^{s,t}\cong 
\bigoplus_{\sigma\in \Lambda_s}H^t_{\Gamma_{\sigma}}(X_{\infty},X_{\infty,s};\FF_2) 
\Longrightarrow H^{s+t}_{\Gamma}(X,X_s;\FF_2) \ . 
\end{equation} 
Here $\Lambda_s$ indexes the $s$-dimensional cells in the orbit space of $X_2$ with respect 
to the action of $\Gamma$. The orbit space is a $2$-simplex, i.e. $\Lambda_0$ and $\Lambda_1$ 
contain $3$ elements and $\Lambda_2$ is a singleton. Furthermore $\Gamma_{\sigma}$ 
is the isotropy group of a chosen representative in $X_2$ of the cell $\sigma$ 
in the quotient space. For fixed $s$ all $s$-dimensional cells have isomorphic isotropy groups 
because the $\Gamma$-action on the Bruhat-Tits building is the restriction of a natural action of 
$GL_3(\Z[\frac{1}{2},i])$ on $X_2$ and this action is transitive on the set of $s$-dimensional cells 
(cf. \cite{Buildings}).  

Therefore all isotropy subgroups for the action on $X_2$ are, up to isomorphism,  
subgroups of $SL_3(\Z[i])$ which itself appears as isotropy group of a $0$-dimensional cell in $X_2$. 
The isotropy groups of $1$-dimensional and $2$-dimensional cells are isomorphic to 
well-known congruence subgroups of $SL_3(\Z[i])$. 
By the Soul\'e-Lannes method the fibre $X_{\infty}$ of the projection map $p$ admits a $6$-dimensional 
$SL_3(\Z[i])$-equivariant deformation retract (the space of ``well-rounded hermitean forms" modulo 
arithmetic equivalence)  with compact quotient (cf. \cite{Ash}) and therefore we have 
\begin{equation}\label{E1-finite} 
E_1^{s,t}=0\ \text{unless}\ s=0,1,2,\ 0\leq t\leq 6,\ \text{and}\ \dim_{\F_2}E_1^{s,t}<\infty\ 
\text{for all}\  (s,t)\ . 
\end{equation} 

The $E_1$-term of this spectral sequence should be accessible to machine calculation. 
The spectral sequence will necessarily degenerate at $E_3$ and the calculation of the 
$d_1$-differential and, if necessary the $d_2$-differential, is likely to need human intervention, 
as it was necessary in the case of $SL(3,\Z[\frac{1}{2}])$ (cf. section 3.4 of \cite{sl3}). 
Likewise the calculation of the connecting homomorphism for the mod-$2$ Borel cohomology 
of the pair $(X,X_s)$ is likely to require human intervention. 
\smallskip

\section{Relation to Quillen's conjecture}\label{RQC}

The next result gives gives 2 reformulations of Quillen's conjecture which we had briefly 
discussed in the introduction. The classes $e_{2k-1},e_{2k-1}'$ figuring in part c)  
will be introduced in (\ref{e-classes}) below. 

\begin{thm}\label{QC} Suppose $n\geq 2$. Then the following statements are equivalent.  

a)  $C(n,\Z[\frac{1}{2},i],2)$ holds, i.e. $H^*(GL_n(\Z[\frac{1}{2},i]);\FF_2)$ 
is a free module over $\Z/2[c_1,\ldots,c_n]$ where the $c_i$ are the mod-$2$ Chern classes 
of the tautological $n$-dimensional complex representation of $GL_n(\Z[\frac{1}{2},i])$. 

b) The restriction homomorphism $H^*(GL_n(\Z[\frac{1}{2},i]);\FF_2)\to H^*(D_n(\Z[\frac{1}{2},i]);\FF_2)$   
is injective where $D_n(\Z[\frac{1}{2},i])$ denotes the subgroup of diagonal matrices in 
$GL_n(\Z[\frac{1}{2}])$.  

c)  There are isomorphisms 
$$
H^*(GL_n(\Z[{1\over 2},i]);\FF_2)\cong \FF_2[c_1,\ldots,c_n]\otimes 
E(e_1,e_1',\ldots, e_{2n-1},e_{2n-1}')
$$ 
where the classes $c_k$ are the Chern classes of the tautological $n$-dimensional complex 
representation of $GL_n(\Z[\frac{1}{2},i])$ and the classes 
$e_{2k-1},e_{2k-1}'$ are of cohomological degree $2k-1$ for $k=1,\ldots,n$.   
\end{thm} 

\begin{proof}  It is trivial that $(c)$ implies $(a)$.   

In order to show that (a) implies (b) we observe that $D_n(\Z[\frac{1}{2},i])$ is the centralizer of the unique, 
up to conjugacy, maximal elementary abelian $2$-subgroup $E_n$ of  $GL_n(\Z[\frac{1}{2},i])$ given by 
the subgroup of diagonal matrices of order $2$.  Now consider the top Dickson invariant $\omega$ in 
$H^*(BGL_n(\C);\FF_2)$, i.e. the class whose restriction to $H^*B(\prod_{i=1}^nGL_1(\C));\F_2)$ 
is the product of all non-trivial classes of degree $2$. The image of $\omega$ in 
$H^*(GL_n(\Z[{1\over 2},i]);\F_2)$ restricts trivially to the cohomology of all elementary abelian 
$2$-subgroups $E$ of $GL_n(\Z[{1\over 2},i])$ of rank less than $n$. 
If (a) holds then the image of $\omega$ is not a zero divisor in $H^*(GL_n(\Z[{1\over 2},i]);\F_2)$ 
and hence Corollary I.5.8 of \cite{HLS3} implies that the restriction to the centralizer of $E_n$ 
is injective. 

The implication $(b)\Rightarrow (c)$ follows from Proposition \ref{res} below. 
\end{proof}
\smallskip 

Before we go on we introduce the classes $e_{2k-1}$ and $e_{2k-1}'$. As in the case of $GL_2$ they 
are obtained from Quillen's classes $q_{2k-1}\in H^{2k-1}(GL_n(\FF_5);\FF_2)$ \cite{Q2} 
which restrict  in the cohomology of diagonal matrices in $\FF_5$  
to the symmetrization of the class $y_1\ldots y_{k-1}x_k$ where $y_k$ is of cohomological degree $2$ 
corresponding to the $k$-th factor in the product $\prod_{k=1}^n\FF_5^{\times}$ and $x_k$ is of 
cohomolo{\-}gical degree $1$ of the same factor. Then we define 
\begin{equation}\label{e-classes}
e_{2k-1}:=\pi^*(q_{2k-1}), \ \ \ e_{2k-1}':=\pi'^*(q_{2k-1}) 
\end{equation}
where $\pi,\pi'$ are the two ring homomorphisms $\Z[\frac{1}{2},i]\to \FF_5$ 
with $\pi$ sending $i$ to $3$ and $\pi'$ sending $i$ to $2$ which we considered earlier in section 
\ref{CSSGamma}. If we identify the mod-$2$ cohomology $H^*(D_n(\Z[{1\over 2},i]);\FF_2)$ 
with $\FF_2[y_1,\ldots y_n]\otimes E(x_1,x_1'\ldots,x_n,x_n')$ with $y_k$, $k=1,\ldots,n$ 
of degree $2$ and $x_k,x_k'$, $k=1,\ldots,n$ of degree $1$ where as before we choose $x_k$ and $x_k'$ 
to be the basis which is dual to the basis of the $k$-th factor in 
$$D_n(\Z[{1\over 2},i])/D_n(\Z[{1\over 2},i])^2\cong 
\Big(\Z[{1\over 2},i]^{\times}/(\Z[{1\over 2},i]^{\times})^2\Big)^n
$$  
given by the classes of $i$ and $1+i$ then we get the following lemma which generalizes 
(\ref{restriction}) and whose straighforward proof we leave to the reader.

\begin{lem}\label{e-classes1} The class $e_{2k-1}$ restricts in the cohomology of the subgroup of 
diagonal matrices $H^*(D_n(\Z[\frac{1}{2},i];\FF_2))$ 
to the symmetrization of $y_1\ldots y_{k-1}x_k$ and the class $e_{2k-1}'$ restricts 
to the symmetrization of $y_1\ldots y_{k-1}(x_k+x_k')$.  \qed 
\end{lem}

The following result determines the image of the restriction homomorphism and  
shows that (b) implies (c) in Theorem \ref{QC}. It resembles results of Mitchell \cite{Mitchell} for 
$GL_n(\Z[{1\over 2}])$ for $p=2$ and of Anton \cite{An1} for $GL_n(\Z[{1\over 3},\zeta_3])$ for $p=3$.  
\smallskip 

\begin{prop}\label{res} Let $n\geq 1$ be an integer. The image of the restriction map 
$$
i^*:H^*(GL_n(\Z[{1\over 2},i]);\F_2)\to H^*(D_n(\Z[{1\over 2},i]);\F_2)\cong \F_2[y_1,\ldots y_n]
\otimes E(x_1,x_1'\ldots,x_n,x_n') 
$$ 
is isomorphic to  
$$
\F_2[c_1,\ldots c_n]\otimes E(e_1,e_1',\ldots,e_{2n-1},e_{2n-1}')\ . 
$$  
\end{prop} 

Here we have identified the Chern classes $c_i$ and the classes $e_{2i-1}$ and $e_{2i-1}'$ 
with their image via $i^*$.  The images of the elements $c_i$ are, of course, the elementary symmetric 
polynomials in the $y_i$ and the images of the classes $e_{2i-1}$ and $e_{2i-1}'$ have been 
determined in Lemma \ref{e-classes1}. We remark that even though $i^*$ need not be injective, 
it is injective on the subalgebra of $H^*(GL_n(\Z[{1\over 2},i]);\F_2)$ 
generated by the classes $c_i$, $e_{2i-1}$ and $e_{2i-1}'$, $1\leq i\leq n$.   
\smallskip

This proposition is an analogue of Proposition 3.6 of \cite{An2}. Its proof uses crucially condition 
(\ref{image}) below, which also plays a central role in \cite{An2}.    

\begin{proof}  
In this proof we denote the subalgebra 
$$
\F_2[c_1,\ldots c_n]\otimes E(e_1,e_1',\ldots,e_{2n-1},e_{2n-1}')\ . 
$$  
of $H^*(D_n(\Z[{1\over 2},i]);\F_2)$ by $C_n$ and the image of the restriction map by $B_n$. 
We need to show that $B_n=C_n$. This is trivial if $n=1$ and for $n=2$ this follows from Theorem 1 of 
\cite{Weiss} (cf. (\ref{n=2}) and (\ref{rest}) and Lemma \ref{e-classes1}). 

The classes $c_1,\ldots,c_n$ are in $B_n$ as images of the Chern classes with the same name and 
the classes $e_1,\ldots e_{2n-1}$, $e_1',\ldots e_{2n-1}'$ are in $B_n$ by Lemma \ref{e-classes1}. 
Therefore we have $C_n\subset B_n$. We will show $B_n\subset C_n$ for $n\geq 2$ 
by induction on $n$. This will be done in three steps. 

1. From the inclusions 
$$
\begin{array}{rlc}
GL_{n-2}(\Z[{1\over 2},i])\times GL_2(\Z[{1\over 2},i])&\subset & GL_n(\Z[{1\over 2},i]) \\
\\
GL_{n-1}(\Z[{1\over 2},i])\times GL_1(\Z[{1\over 2},i])&\subset & GL_n(\Z[{1\over 2},i]) \\
\end{array}
$$ 
given by matrix block sum and the identifications of  
$D_{n-2}(\myLambda)\times D_2(\myLambda)$ with $D_n(\myLambda)$ 
and of $D_{n-1}(\myLambda)\times D_1(\myLambda)$ with $D_n(\myLambda)$
we see that 
$$
B_n\subset B_{n-1}\otimes B_1\cap B_{n-2}\otimes B_2
$$ 
and by induction hypothesis the latter subalgebra is equal to 
$$
C_{n-1}\otimes C_1\cap C_{n-2}\otimes C_2\ ,  
$$ 
in particular we have 
\begin{equation}\label{BnC}
B_n\subset C_{n-1}\otimes C_1\cap C_{n-2}\otimes C_2 \ . 
\end{equation}

2. The monomial basis in  
$$
H^*(D_n(\Z[{1\over 2},i]);\F_2)\cong \FF_2[y_1,\ldots,y_n]\otimes E(x_1,\ldots,x_n,x_1',\ldots,x_n')
$$ 
is in bijection with the set $S(n)$ of sequences 
$$
I=(a_1,\varepsilon_{1,1},\varepsilon_{2,1},\ldots,a_n,\varepsilon_{1,n},\varepsilon_{2,n}) 
$$
where the $a_i$ are integers $\geq 0$ and $\varepsilon_{i,j}\in \{0,1\}$ for $i=1,2$ and $1\leq j\leq n$. 
More precisely to $I$ we associate the monomial 
$$
y^I:=y_1^{a_1}\ldots y_n^{a_n}x_1^{\varepsilon_{1,1}}\ldots x_n^{\varepsilon_{1,n}}
{x_1'}^{\varepsilon_{2,1}}\ldots {x_n'}^{\varepsilon_{2,n}}\ . 
$$  
We equip $S(n)$ with the lexicographical order and denote it by $<_n$.  
This order has the property that for each $1\leq k<n$ 
it agrees with the lexicographical order on $S(k)\times S(n-k)$ if $S(k)$ and $S(n-k)$ are equipped with 
the orders $<_k$ and $<_{n-k}$ and $S(n)$ is identified with $S(k)\times S(n-k)$ via concetanation of 
sequences. 

In the sequel we replace the symmetrizations of the elements $y_1\ldots y_{i-1}(x_i+x_i')$, $i=1,\ldots,n$,   
by the symmetrization of $y_1\ldots y_{i-1}x_i'$ and by abuse of notation we continue to denote them by 
$e_{2i-1}'$. This does not change the subalgebra $C_n$. This subalgebra 
$$
\F_2[c_1,\ldots c_n]\otimes E(e_1,e_1',\ldots,e_{2n-1},e_{2n-1}')\subset 
\FF_2[y_1,\ldots,y_n]\otimes E(x_1,\ldots,x_n,x_1',\ldots,x_n')
$$  
has a monomial basis which is in bijection with the set $T(n)$ of sequences 
$$
K=(k_1,\ldots,k_n;\phi_{1,1},\ldots,\phi_{1,n};\phi_{2,1}\ldots,\phi_{2,n})
$$ 
where the $k_i$ are integers $\geq 0$ and $\phi_{i,j}\in \{0,1\}$ for $i=1,2$ and $1\leq j\leq n$. 
More precisely to $K$ we associate the monomial 
$$
c^K:=c_1^{k_1}\ldots c_n^{k_n}e_1^{\phi_{1,1}}\ldots e_n^{\phi_{1,n}}
e_1^{\phi_{2,1}}\ldots e_n^{\phi_{2,n}}\ . 
$$ 
We define a map 
$$
\alpha: T(n)\to S(n) 
$$
by associating to $K\in T(n)$ the largest monomial in $S(n)$ which occurs in the 
decomposition of $c^K$ as linear combination of elements $x^I$ with $I\in S(n)$. 
The proof of the following result is elementary and is left to the reader. 

\begin{lem} 
The map $\a$ is explicitly given by 
$$ 
\a((k_1,\ldots,k_n;\phi_{1,1},\ldots,\phi_{1,n};\phi_{2,1}\ldots,\phi_{2,n}))
=(a_1,\varepsilon_{1,1},\varepsilon_{2,1},\ldots,a_n,\varepsilon_{1,n},\varepsilon_{2,n})
$$ 
with 
\begin{align*}
a_1=&\ k_1+\ldots k_n+\sum_{i=1}^2(\phi_{i,2}+\ldots \phi_{i,n}) \\ 
a_2=&\ k_2+\ldots k_n+\sum_{i=1}^2(\phi_{i,3}+\ldots \phi_{i,n}) \\ 
\ldots &\ \ldots \\
a_j=&\ k_j+\ldots k_n+\sum_{i=1}^2(\phi_{i,j+1}+\ldots \phi_{i,n}) \\ 
\ldots &\ \ldots \\
a_n=&\ k_n \\
\varepsilon_{i,j}= &\ \phi_{i,j}, \ \ 1\leq j\leq n,\ i=1,2 \ .  \ \ \ \ \ \ \ \ \ \ \ \ \ \ \ \ \ \ \ \ \ \ \ \ \ \ \  \qed \\
\end{align*}  
\end{lem} 
 
From this lemma it is obvious that $\a$ is injective and a sequence 
$$
I=(a_1,\varepsilon_{1,1},\varepsilon_{2,1},\ldots,a_n,\varepsilon_{1,n},\varepsilon_{2,n}) \in S(n)
$$  
is in the image of $\a$ if and only if  we have 
\begin{equation}\label{image}
a_j-a_{j+1}\geq \varepsilon_{1,j+1}+\varepsilon_{2,j+1}\ \ \text{for all}\ \ 1\leq j<n \ . 
\end{equation} 
In particular, if an element $x$ is in $C_n$ then the maximal sequence which appears 
in the decomposition of $x$ as a linear combination of the monomials $x^I$ 
with $I\in S(n)$ satisfies (\ref{image}) for all $1\leq j<n$. 
Likewise, if $x$ is in $C_{i}\otimes C_{n-i}$
then this maximal sequence is equal to the maximal sequence which appears 
in the decomposition of $x$ as a linear combination of the monomials $x^I$ 
with $I\in S(k)\times S(n-k)$ and hence it satisfies (\ref{image}) for all $1\leq j<i$ and $i+1\leq j<n$.  

3. Now let $x$ be a homogeneous element of $B_n$ and let $I_0$ be the maximal sequence in $S(n)$ 
appearing in the decomposition of $x$ as a linear combination of the monomials $x^I$ 
with $I\in S(n)$.  By (\ref{BnC}) we have $x\in C_{n-1}\otimes C_1$ and $x\in C_{n-2}\otimes C_2$,   
and $I_0$ remains the maximal sequence in $S(n-1)\times S(1)$ resp. $S(n-2)\times S(2)$ 
appearing in the decomposition of $x$ as a linear combination of the monomials $x^I$ 
with $I\in S(n-1)\times S(1)$ resp. $I\in S(n-2)\times S(2)$. 
Hence $I_0$ satisfies conditions (\ref{image}) for $1\leq j<n-1$ resp. $1\leq j<n-2$ and $j=n-1$. 
In particular condition condition (\ref{image}) holds for all $1\leq j<n$ and therefore there exists 
$K_0\in T(n)$ such that $\a(K_0)=I_0$. Then $x-c^{K_0}$ is still in $B_n$ 
and the maximal sequence appearing in the decomposition of $x-c^{K_0}$ 
is smaller than that of $x$. By iterating this procedure we see that $x$ belongs to $C_n$.  
\end{proof}

Finally we relate $C(3,\Z[\frac{1}{2},i],2)$ to the behaviour of the restriction homomorphism 
$$
H^*(\Gamma;\F_2)\to H^*(C_{\Gamma}(E_2);\F_2) \ . 
$$ 
For this we observe that the subgroups $\Gamma=SL_3(\Z[\frac{1}{2},i])$ and the center 
$Z\cong \Z[\frac{1}{2},i]^{\times}$ of $GL_3(\Z[\frac{1}{2},i])$ have trivial intersection 
and their product is the kernel of the homomorphism 
$$
GL_3(\Z[\frac{1}{2},i])\to (\Z[\frac{1}{2},i])^{\times}\to (\Z[\frac{1}{2},i])^{\times}/(\Z[\frac{1}{2},i])^{\times})^3\cong \Z/3 
$$ 
given as the composition of the determinant with the natural quotient map. 
Therefore the spectral sequence of the extension 
$$
1\to SL_3(\Z[\frac{1}{2},i])\times Z\to GL_3(\Z[\frac{1}{2},i])\to \Z/3\to 1
$$ 
gives an isomorphism 
\begin{equation}\label{index3}
H^*(GL_3(\Z[\frac{1}{2},i]);\FF_2)\cong 
\big(H^*(SL_3(\Z[\frac{1}{2},i]);\F_2)\otimes H^*(Z;\FF_2)\big)^{\Z/3} \ . 
\end{equation} 
\smallskip

\begin{prop}\label{C3+main} The conjecture $C(3,\Z[{1\over 2},i],2)$ holds if and and only if either   
\smallskip 

a) $H^*(SL_3(\Z[{1\over 2},i]);\F_2)\cong \F_2[b_2,b_3]\otimes E(d_3,d_3',d_5,d_5')$ or 
\smallskip 

b) the kernel of the map $\psi$ of Theorem \ref{mainthm} is a finite dimensional vector space 
for which the action of $\Z/3\cong (\Z[\frac{1}{2},i])^{\times}/(\Z[\frac{1}{2},i])^{\times})^3$ 
has trivial invariants. 
\end{prop} 

\begin{proof} The quotient $\Z/3\cong (\Z[\frac{1}{2},i])^{\times}/(\Z[\frac{1}{2},i])^{\times})^3$ 
acts clearly trivially on $H^*(Z;\FF_2)$ 
and on the image of the homomorphism $\varphi$ of Theorem \ref{mainthm}. 
Hence, the corollary follows immediately from (\ref{index3}) and Theorem \ref{mainthm}. 
\end{proof}

\end{document}